\documentclass[letterpaper,11pt]{amsart}

\usepackage[all]{xy}                        %

\CompileMatrices                            

\UseTips                                    

\input xypic
\usepackage[bookmarks=true]{hyperref}       

\usepackage{amssymb,latexsym,amsmath,amscd}
\usepackage{xspace}

\usepackage{graphicx}

\reversemarginpar

\theoremstyle{plain}
\newtheorem{theorem}{Theorem}[section]
\newtheorem*{theorem*}{Theorem}
\newtheorem{proposition}[theorem]{Proposition}
\newtheorem{corollary}[theorem]{Corollary}
\newtheorem{lemma}[theorem]{Lemma}

\theoremstyle{definition}
\newtheorem{definition}[theorem]{Definition}

\newtheorem{remark}[theorem]{Remark}

\newcommand{\enm}[1]{\ensuremath{#1}}          %
\newcommand{\op}[1]{\operatorname{#1}}
\newcommand{\cal}[1]{\mathcal{#1}}

\newcommand{\CC}{\enm{\mathbb{C}}}

\newcommand{\PP}{\enm{\mathbb{P}}}

\newcommand{\Oo}{\enm{\cal{O}}}

\newcommand{\Ww}{\enm{\cal{W}}}

\renewcommand{\phi}{\varphi}
\renewcommand{\theta}{\vartheta}
\renewcommand{\epsilon}{\varepsilon}

\newcommand{\Pic}{\op{Pic}}

\newcommand{\Ext}{\op{Ext}}

\newcommand{\old}[1]{}

\begin{document}

\title[Moduli of Stable Sheaves on a Smooth Quadric]{Moduli of Stable Sheaves\\
on a Smooth Quadric in $\PP_3$}
\author{Sukmoon Huh}
\address{CIRM \\
Fondazione Bruno Kessler \\
Via Sommarive, 14-Povo\\
38100-Trento}
\email{sukmoon.huh@math.unizh.ch}
\begin{abstract}
We prove that the moduli space of stable sheaves of rank 2 with a certain Chern classes on a smooth quadric $Q$ in $\PP_3$, is isomorphic to $\PP_3$. Using this identification, we give a new proof that a certain Brill-Noether locus on a non-hyperelliptic curve of genus 4, is isomorphic to the \textit{Donagi-Izadi cubic threefold}.
\end{abstract}

\maketitle

\section{Introduction}
In \cite{OPP}, the geometry of the Brill-Noether loci $\Ww^r$ of the
moduli spaces $SU_C(2,K_C)$ of semi-stable vector bundles of rank 2
with canonical determinant over curves with small genus, was investigated. In \cite{Huh}, the explicit description of $\Ww^1$ when
$g(C)=3$, was rediscovered, using the rational map from the moduli
space $\overline{M}(1,2)$ of stable sheaves of rank 2 with Chern
classes (1,2) over $\PP_2$ to $\Ww^1$, defined by the restriction. In
this article, we use the same method as in \cite{Huh} to rediscover
the geometry of $\Ww^2$ stated in \cite{OPP}, using the rational
restriction map
$$\Phi : \overline{M}(2) \dashrightarrow \Ww^2,$$
sending $E$ to $E|_C$, where $\overline{M}(k)$ is the moduli space of semi-stable sheaves of
rank 2 with the Chern classes $c_1=(1,1)$ and $c_2=k$ over a smooth
quadric surface $Q\subset \PP_3$, containing $C$, and $\Ww^2$ is a Brill-Noether
locus of a non-hyperelliptic curve of genus 4. This rational map
makes sense because $C$ is canonically embedded into $\PP_3$ and
there exists the unique quadric containing it. We will deal with the case when $Q$ is smooth.

First, we give the explicit description of $\overline{M}(2)$. We construct a morphism $\Psi: \overline{M}(2) \rightarrow \PP_3$, sending $E$ to the point $p_E$ which is passed by the lines containing zeros of the sections of $E$. In fact, $\Psi$ can be proven to be an isomorphism.

Second, we study the jumping conics of $E\in \overline{M}(2)$. In fact, the isomorphism $\Psi$ can be redefined by sending $E$ to the set of jumping conics of $E$, which is a hyperplane in $\PP_3^*$.

Last, we investigate the restriction map $\Phi : \overline{M}(2) \dashrightarrow \Ww^2$. We prove that this map is birational and given by the complete linear system $|I_C(3)|$. This will give us an easy proof that $\Ww^2$ is isomorphic to the \textit{Donagi-Izadi cubic threefold}. If we composite $\Phi$ with the projection $\Ww^2 \dashrightarrow \PP_3^*$ at the unique point of $\Ww^3$, we have an isomorphism of $\overline{M}(2)$ with $\PP_3^*$. We also describe this isomorphism in terms of sheaves in $\overline{M}(2)$.

The work in this article was done during my stay at the CIRM in Trento. I am grateful to the institution for the hospitality and support.

\section{Description of $\overline{M}(2)$}
Let $Q$ be a smooth quadric in $\PP_3$ over $\CC$ and $\overline{M}(k)$ be the moduli space of semi-stable sheaves of rank 2 with the Chern classes $c_1=\Oo_Q(1,1)$ and $c_2=k$ on $Q$ with respect to the ample line bundle $H=\Oo_Q(1,1)$.
We know, by the Bogomolov theorem, that $\overline{M}(k)=\emptyset$ for $k\leq 0$. For the case when $k=1$, the following statement can be easily derived.

\begin{lemma}\label{1}
$\overline{M}(1)$ is the one-point space with the strictly semi-stable bundle,
$$E_0=\Oo_Q(1,0)\oplus \Oo_Q(0,1).$$
\end{lemma}
\begin{proof}
Let $E\in \overline{M}(1)$ be a stable bundle and then, from $\chi(E)=4$ and the stability condition, we have
$$0\rightarrow \Oo_Q \rightarrow E \rightarrow I_p(1,1) \rightarrow 0,$$
where $I_p$ is the ideal sheaf of a point $p\in Q$. If we tensor the sequence with $\Oo_Q(-1,0)$, it can be shown that $h^0(E(-1,0))=1$, which is a contradiction to the stability of $E$. If $E$ is strictly semi-stable, it should be fitted into the following sequence,
$$0\rightarrow \Oo_Q(a,1-a) \rightarrow E \rightarrow \Oo_Q(1-a,a) \rightarrow 0,
$$
where $a=0$ or $1$. Since $\Ext^1(\Oo_Q(a,1-a),\Oo_Q(1-a,a))$ is trivial, so $E$ must be the direct sum of $\Oo_Q(1,0)$ and $\Oo_Q(0,1)$.
\end{proof}

Now let us deal with the case $k=2$. Note that the stability and semi-stability conditions are equivalent. In particular, $\overline{M}(2)$ is a projective space whose points correspond to isomorphism classes of stable sheaves on $Q$ with given numerical invariants. From the standard computation, we can have $\Ext^2(E,E)=0$ for any $E\in \overline{M}(2)$. Thus $\overline{M}(2)$ is smooth.

Let $E$ be a sheaf in $\overline{M}(2)$. For a subsheaf $\Oo_Q(a,b)$ of $E$, we have $a+b<1$ by the stability condition. Since $\chi(E)=3$ and $$h^2(E)=h^0(E^*(-2,-2))=h^0(E(-3,-3))=0,$$so we have $h^0(E)>0$.
Hence, $E$ can be obtained from the following extension,
\begin{equation}\label{ext}
0\rightarrow \Oo_Q \rightarrow E \rightarrow I_Z(1,1) \rightarrow 0,
\end{equation}
where $Z$ is a zero-dimensional subscheme of $Q$ with length 2. In particular, we have $h^0(E)=3$ and $h^1(E)=0$. Assume that there exists a line $l$ on $Q$ containing $Z$ with the ideal sheaf $\Oo_Q(-1,0)$. If we tensor the sequence (\ref{ext}) with $\Oo_Q(0,-1)$, it can be easily checked that $h^0(E(0,-1))=0$ contradicting to the stability of $E$. Thus, the line $l$ containing $Z$ should intersect with $Q$ only at $Z$.

\begin{lemma}
The sheaves in the extension (\ref{ext}), are all stable if $Z$ is not contained in any line on $Q$.
\end{lemma}
\begin{proof}
From the condition on the numerical invariants of $E$, the only possibility of the sub-bundle $\Oo_Q(a,b)\subset E$ is $\Oo_Q$ or $\Oo_Q(a,1-a)$ with $a=0,1$. The second case is impossible due to the condition on $Z$.
\end{proof}

Note that $\PP(Z):=\PP \Ext^1(I_Z(1,1), \Oo_Q)\simeq \PP_1$. From the isomorphism
$$\Ext^1(I_Z(1,1), \Oo_Q)\simeq \Oo_{z_1}\oplus \Oo_{z_2},$$
we can denote $(c_1,c_2)$ for the coordinates of $\PP(Z)$. As in the case of the projective plane \cite{Hulek}, we can prove the following lemma.

\begin{lemma}
An extension $(c_1,c_2)$ gives a bundle if and only if all $c_i\not=0$.
\end{lemma}
\begin{proof}
An easy application of \cite{Hulsbergen}.
\end{proof}

\begin{lemma}\label{non}
The set of non-bundles in $\overline{M}(2)$ is parametrized by $Q\subset \PP_3$.
\end{lemma}
\begin{proof}
Let $E$ be a non-bundle in $\overline{M}(2)$, then we have the following exact sequence,
\begin{equation}
0\rightarrow E \rightarrow E^{**} \rightarrow \Oo_Z \rightarrow 0,
\end{equation}
where $Z$ is a zero-dimensional subscheme of $Q$. Since $E^{**}$ is semi-stable, we have $E^{**}\simeq E_0=\Oo_Q(1,0)\oplus \Oo_Q(0,1)$ by [lemma\ref{1}]. In particular, the length of $Z$ is 1.

Now let $f: E_0\twoheadrightarrow \Oo_{p_E}$ be a surjection with the kernel $E$, where $p_E$ is a point in $Q$. We can denote $f$ by $(f_1,f_2)$, where $f_i : \Oo_Q(i-1,i) \rightarrow \Oo_{p_E}$, i.e. the parametrization of such surjections is $\PP_1$ with the coordinates $(f_1,f_2)$. If $f_1=0$, then $\ker (f)$ is decomposed to $\Oo_Q(1,0)\oplus I_{p_E}(0,1)$, which is not stable. Similarly, we have an unstable kernel when $f_2=0$. Let us assume that $f_i\not=0$ for all $i$. Then we have the following diagram,

\begin{equation}
\xymatrix{ & 0 \ar[d] & 0 \ar[d] & 0 \ar[d] \\
             0 \ar[r] & I_{p_E}(1,0) \ar[d] \ar[r]
             &E \ar[d] \ar[r] &
             \Oo_Q(0,1) \ar[d] \ar[r] & 0\\
             0 \ar[r] &\Oo_Q(1,0) \ar[d] \ar[r]^{s} & E_0
                          \ar[d] \ar[r] & \Oo_Q(0,1) \ar[d] \ar[r] &0\\
             0\ar[r]&\Oo_{p_E} \ar[d]\ar[r] & \Oo_{p_E} \ar[d] \ar[r] & 0  \\
             &0& 0,}
\end{equation}
where $E$ is a non-trivial extension. It is easily checked that this non-trivial sheaf is stable. Since the dimension of $\Ext^1(\Oo_Q(0,1), I_{p_E}(1,0))$ is 1, the non-trivial extension $E$ is uniquely determined, i.e. we have only one non-bundle on $Q$ associated to a point $p_E\in Q$. Hence, we get the assertion.
\end{proof}

Now let $E$ be a stable bundle in $\overline{M}(2)$ and $s$ be a section of $E$.
From $s$, we have an exact sequence of type (\ref{ext}) and so can consider a morphism
$$\psi_E : \PP H^0(E)\simeq \PP_2 \rightarrow Gr(1,3)\subset \PP_5,$$
sending a section $s$ of $E$ to the line containing $Z$ in $\PP_3$.

\begin{remark}
In the previous proof, if we choose a section of $E_0$ whose zero is $q\not= p_E$, then we have an exact sequence
$$0\rightarrow \Oo_Q \rightarrow E \rightarrow I_Z(1,1) \rightarrow 0,$$
where $Z=\{p_E,q\}$. In particular, for a non-bundle $E$, the image of $\phi_E$ is the plane in $Gr(1,3)$ whose points correspond to lines passing through the singularity point $p_E$.
\end{remark}

Let $E$ be a stable bundle in $\overline{M}(2)$ with the exact sequence (\ref{ext}).
\begin{lemma}\label{inj}
The determinant map
$$\lambda_E : \wedge^2H^0(E) \rightarrow H^0(\Oo_Q(1,1)),$$
is injective.
\end{lemma}
\begin{proof}
Let us assume that $\lambda_E$ is not injective. Since every element in $\wedge^2 H^0(E)$ is decomposable, there exist two sections $s_1$ and $s_2$ of $E$ such that $s_1\wedge s_2$ is a non-trivial element in $\ker(\lambda_E)$. Then $s_1$ and $s_2$ generate a subsheaf $F$ of $E$ with $h^0(F)\geq 2$. Hence, $F$ is of the form $I_{Z'}(a,b)$, where $Z'$ is a 0-cycle on $Q$ and $a,b\geq 0$. From the stability condition of $E$, we have $F\simeq \Oo_Q(a,1-a)$ with $a=0$ or $1$, which is not possible since $Z$ is not contained in any line on $Q$.
\end{proof}

Let $p_E$ be the point in $\PP_3\simeq \PP H^0(\Oo_Q(1,1))^*$ corresponding to the cokernel of $\lambda$ and consider the following exact sequence,
$$0\rightarrow H^0(I_Z(1,1))\rightarrow H^0(\Oo_Q(1,1))\rightarrow H^0(\Oo_Z)\rightarrow 0.$$
In the previous lemma, $H^0(E)$ can be expressed as the direct sum of $H^0(\Oo_Q)$ and $H^0(I_Z(1,1))$, which would give the following identification,
\begin{equation}
\wedge^2 H^0(E) \simeq [H^0(\Oo_Q)\otimes H^0(I_Z(1,1))]\oplus [\wedge^2 H^0(I_Z(1,1))].
\end{equation}
From this identification, clearly we have,
\begin{equation}
H^0(I_Z(1,1))\subset \lambda_E(\wedge^2H^0(E)).
\end{equation}
In other words, the dual of the cokernel of $\lambda$ is contained in $H^0(\Oo_Z)^*$. Note that $\PP H^0(\Oo_Z)^*$ is a line in $\PP_3$ passing through $Z$. Thus the line passing through $Z$, also contains $p_E$. With the previous remark, we get the following lemma.

\begin{lemma}
$\psi_E$ is a linear embedding of $\PP H^0(E)$ into $\PP_5$. Moreover, the image corresponds to the set of lines passing through one point $p_E$ in $\PP_3$.
\end{lemma}

Now, let us define a map
$$\Psi : \overline{M}(2) \rightarrow \PP_3,$$
sending $E$ to $p_E$, where $p_E$ is the unique point in $\PP_3$, passed by the lines in the image of $\phi_E$.

\begin{proposition}
$\Psi: \overline{M}(2) \rightarrow \PP_3$ is an isomorphism.
\end{proposition}
\begin{proof}
Let $p$ be a point in $\PP_3$. If $Z$ is a 0-cycle on $Q$ such that $p\in Z$ and the line passing through $Z$, is not contained in $Q$, then there exists $E\in \PP(Z)$ for which $p_E=p$. Thus, $\Phi$ is surjective.

Moreover, assume that $p$ is not in $Q$. If we take the projection $\pi_p: \PP_3 \dashrightarrow \PP_2$ from $p$, then the restriction map $\pi_p : Q \rightarrow \PP_2$ is a finite morphism of degree 2. Again, if we take the pull-back of $\Omega_{\PP_2}(2)$ on $Q$, we get a stable vector bundle on $Q$ with the Chern classes $c_1=\Oo_Q(1,1)$ and $c_2=2$ (this will be proved later in [lemma\ref{stable}]). Clearly, this defines an inverse map of $\Psi$. Hence, $\Psi$ is a birational morphism and it is an isomorphism over the stable vector bundles.

Since $\Phi$ is also an isomorphism over the non-bundles from [lemma\ref{non}], $\Phi$ is an isomorphism.
\end{proof}

\begin{remark}
From the identification of $\overline{M}(2)$ with $\PP_3$, we know that $\PP(Z)$ is exactly the secant line of $Q$ in $\PP_3$ passing through $Z$.
\end{remark}

\begin{remark}
Alexander Kuznetsov pointed out that the moduli space $\overline{M}(2;1,1,1)$ of stable sheaves on $\PP_3$ of rank 3 with the Chern classes $(c_1,c_2,c_3)=(1,1,1)$ is isomorphic to $\PP_3$, whose points correspond to the cokernels $E$ of $\Oo_{\PP_3} \rightarrow T_{\PP_3}(-1)$, where $T_{\PP_3}$ is the tangent bundle of $\PP_3$. Note that $h^0(\PP_3, T_{\PP_3}(-1))=4$. The restriction map from $\overline{M}(2;1,1,1) \rightarrow \overline{M}(2)$ turns out to be an isomorphism.
\end{remark}

\section{Jumping Conics}
\begin{definition}
Let $E$ be a stable sheaf in $\overline{M}(2)$ and $H$ be a hyperplane in $\PP_3$. A conic $C_H=Q\cap H$ on $Q$, is called \textit{a jumping conic} of $E$ if the splitting type of $E|_{C_H}$ is different from the generic splitting type.
\end{definition}

Assume that $E$ is locally free. If $H$ is a hyperplane, which does not contain $p_E$, then we can choose a line $l$ through $p_E$, with the unique intersection point with $C_H$, say $q$. Let $Z=l\cap Q$ and then $E$ is fitted into the exact sequence (\ref{ext}), i.e. $E\in l=\PP(Z)$. If we tensor the sequence with $\Oo_{C_H}$, then $E|_{C_H}$ lies in $\Ext^1 (\Oo_{C_H}(1-q),\Oo_{C_H}(q))=0$, i.e.
\begin{equation}
E|_{C_H}\simeq \Oo_{C_H}(p)\oplus \Oo_{C_H}(1-p).
\end{equation}
If $H$ contains $p_E$, let us choose a line $l\subset H$ containing $p_E$. If we let $Z=l\cap Q$ again, we have $E\in l=\PP(Z)$. But in this case, $E|_{C_H}$ lies in $\Ext^1 (\Oo_{C_H}, \Oo_{C_H}(1))=0$,i.e.
\begin{equation}
E|_{C_H}\simeq \Oo_{C_H}(1)\oplus\Oo_{C_H}.
\end{equation}

Similarly, when $E$ is not locally free, we can derive the same result.

Hence, all jumping conics can be characterized by $h^0(E(-1)|_{C_H})\not= 0$.
\begin{proposition}
For a stable bundle $E\in \overline{M}(2)$, the set of jumping conics $C_H$, form a hyperplane $H_E\subset \PP_3^*$ corresponding to $p_E\in \PP_3$. In particular, $\Psi$ can be also defined by sending $E$ to the set of jumping conics of $E$.
\end{proposition}

\section{The Restriction Map}
Let $C$ be a non-hyperelliptic curve of genus 4 and it is embedded
into $\PP H^0(K_C)^*\simeq \PP_3$ and there is the unique quadric
surface $Q\subset \PP_3$ containing $C$. Let $g_3^1$ and $h_3^1$ be the two trigonal line bundles in $\Theta\subset  \Pic^3(C)$ such that $g_3^1 \otimes h_3^1=\Oo_C(K_C)$. $Q$ is smooth if and only if $|g_3^1|\not= |h_3^1|$. From now on, we assume that $Q$ is smooth. Let $SU_C(2,K_C)$ be the
moduli space of semi-stable vector bundles of rank 2 with canonical
determinant over $C$. Let $\Ww^r$ be the closure of the following
set
\begin{equation}
\{ E \in SU_C(2,K_C) ~|~ h^0(C,E)\geq r+1\}.
\end{equation}
Then we have the following inclusions \cite{OPP} on the
Brill-Noether loci,
\begin{equation}
SU_C(2, K_C) \supset \Ww \supset \Ww^1 \supset \Ww^2 \supset \Ww^3
\supset \Ww^4 =\emptyset ,
\end{equation}
where $\Ww=\Ww^0$. Many geometric descriptions of these
Brill-Noether loci have been investigated in \cite{OPP}.

Since $C$ is a divisor of $Q$ with the divisor class $(3,3)$, we have the exact
sequence
\begin{equation}\label{sesc}
0\rightarrow \Oo_Q (-3,-3) \rightarrow \Oo_Q \rightarrow \Oo_C
\rightarrow 0.
\end{equation}
Twisting the sequence(\ref{sesc}) with a stable bundle $E\in \overline{M}(2)$, we obtain that $h^0(C,E|_C)=h^0(Q,E)=3$.

\begin{lemma}
For a stable bundle $E\in \overline{M}(2)$, its restriction to $C$,
$E|_C$, is stable and so we have a rational map
$$\Phi : \overline{M}(2) \dashrightarrow \Ww^2$$
sending $E$ to $E|_C$.
\end{lemma}
\begin{proof}
Since the embedding of $C$ into $\PP_3$ is canonical, the
restriction of $\wedge^2 E \simeq \Oo_Q (1,1)$ is $\Oo_C(K_C)$, i.e.
$\det (E|_C)=\Oo_C(K_C)$. If we tensor the exact sequence
(\ref{ext}) by $\Oo_C$, we have
$$0\rightarrow \Oo_C(D)\rightarrow E|_C \rightarrow
\Oo_C(K_C-D)\rightarrow 0,$$ where $D=Z\cap C$ as a scheme with
$l(D)\leq l(Z)=2$. Suppose that there exists a subbundle
$\Oo_C(D')\subset E|_C$ with $\deg (D')=d' \geq 3$. If the natural
composite $f: \Oo_C(D') \rightarrow \Oo_C(K_C-D)$ is zero, then
$\Oo_C(D')\subset \Oo_C(D)$, which is not possible. Thus $f$ is not
zero and must be injective so that $d'$ can be at most 6. Since
$h^0(E|_C)=3$ and $h^0(\Oo_C(K_C-D'))<2$, so $H^0(\Oo_C(D'))$ is not
trivial. Now we can assume that $D'$ is effective and the zeros of a
section in $H^0(E|_C)$. Note that $H^0(E)\simeq H^0(E|_C)$ and so
every section of $E|_C$ comes as the restriction of a section of $E$. Since the zero
of a section of $E$ has only 2 points as its support, the degree of $D'$ must be
less than 3. Hence, $E|_C$ is stable.
\end{proof}

\begin{lemma}\cite{OPP}
For a general $E$ in $\Ww^2$, the determinant map
$$\lambda_E : \wedge^2 H^0(E) \rightarrow H^0(K_C),$$
is injective.
\end{lemma}
\begin{proof}
As in [lemma\ref{inj}], let us assume that $s_1$ and $s_2$ are two sections of $E$ for which $s_1\wedge s_2$ is a non-trivial element in $\ker (\lambda_E)$. It would generate a sub-bundle $L\subset E$ with $h^0(L)\geq 2$ and $\deg(L)\leq 2$, contrary to the fact that $C$ is non-hyperelliptic.
\end{proof}

Denote by $p_E$, the point in $\PP H^0(K_C)^*\simeq \PP_3$ corresponding to the cokernel of $\lambda_E$. Sending $E$ to $p_E$, would define a map
$$\tau : \Ww^2 \dashrightarrow \PP_3.$$
Moreover we have the following diagram,
$$\xymatrix{ C\ar[r] &\PP H^0(K_C)^*\simeq \PP_3 \ar[d]^{\pi_E}\\
                     &\PP (\wedge^2H^0(E)^*) \simeq \PP_2,}$$
where $\pi_E$ is the projection from $p_E$. From the identification of $\PP H^0(K_C)^*\simeq \PP H^0(\Oo_Q(1,1))^*$ and $\PP ( \wedge^2 H^0(E|_C)^*)\simeq \PP ( \wedge^2 H^0(E)^*)$ for $E$ a stable bundle in $\overline{M}(2)$, we know that $p_{E|_C}$ does not lie on $Q$. Hence, $p_E$ does not lie on $Q$ for general $E\in \Ww^2$.

Note that $\PP (\wedge^2H^0(E)^*)$ is the Grassmannian $Gr(H^0(E), 2)$ and its universal quotient bundle is $\Omega_{\PP_2}(2)$. From the construction, we have $\pi_E^*(\Omega_{\PP_2}(2))|_C\simeq E$ and the Chern classes of $\pi_E^*(\Omega_{\PP_2}(2))|_Q$ is $(c_1,c_2)=(\Oo_Q(1,1),2)$ since $\pi_E:Q\rightarrow \PP_2$ is a finite map of degree 2.

\begin{lemma}\label{stable}
$E_Q:=\pi_E^*(\Omega_{\PP_2}(2))|_Q$ is a stable vector bundle on $Q$.
\end{lemma}
\begin{proof}
If we take the pull-back $(\pi_E|_Q)^*$ to the following sequence,
\begin{equation}
0\rightarrow \Oo_{\PP_2} \rightarrow \Omega_{\PP_2}(2) \rightarrow I_p(1) \rightarrow 0,
\end{equation}
where $I_p$ is the ideal sheaf of a point $p\in \PP_2$, then $E_Q$ is an non-trivial extension class in $\Ext^1(I_Z(1,1), \Oo_Q)$. Any non-trivial element in this extension can be easily checked to be stable.
\end{proof}

\begin{corollary}
The restriction map $\Phi$ is birational.
\end{corollary}

\begin{remark}
Note that $g_3^1$ and $h_3^1$ are isomorphic to $\Oo_Q(1,0)|_C$ and $\Oo_Q(0,1)|_C$. In particular, $E_0|_C=g_3^1\oplus h_3^1$ and $h^0(E_0)=h^0(E_0|_C)=4$. As pointed in \cite{OPP}, $g_3^1\oplus h_3^1$ can be considered as the unique point of $\Ww^3$.
\end{remark}
\begin{remark}
Let $E$ be a non-bundle in $\overline{M}(2)$ and $p$ be its singularity point. If $p\not\in C$, then $E|_C\simeq E_0|_C \in \Ww^3$. If $p\in C$, then $E|_C$ is not torsion-free. In particular, the indeterminacy locus of $\Phi$ is exactly $C\subset \PP_3\simeq \overline{M}(2)$.
\end{remark}

\begin{proposition}
The map $\Phi$ is given by the complete linear system $|I_C(3)|$.
\end{proposition}
\begin{proof}
We know that $\Phi$ is an isomorphism on $\PP_3 \backslash Q$ and sends $Q\backslash C$ to one point $E_0|_C$.
Let $H$ be a general hyperplane in $\PP_3$. The intersection of $H$ with $C$ is 6 points on $Q$, say $P_1, \cdots,P_6$ and these points lie on a conic $C_2=H\cap Q$. The restriction of $\Phi$ to $H$ is not defined on $P_i$'s and maps the other points of $C_2$ to $E_0|_C$.

Let $S_6$ be the blow-up of $H$ at $P_i$'s and then we obtain a morphism $f_H$ from $S_6$ to $\Ww^2$.
\begin{equation}
\xymatrix{ & S_6 \ar[ld] \ar[rd]^{f_H}\\
H \ar@{-->}[rr] && \Ww^2}
\end{equation}
The proper transform of $C_2$ in $S_6$ is a line $l_H$ and $f_H$ contracts $l_H$ to a point. Since $S_6$ is a cubic surface in $\PP_3$, the degree of $f_H$ is 3 and so is the degree of $\Phi$.

Recall that the indeterminacy of $\Phi$ is $C$. Since $h^0(I_C(3))=5$ and $\Ww^2$ is a 3-fold which is not isomorphic to $\PP_3$, so $\Phi$ must be given by the complete linear system $|I_C(3)|$.
\end{proof}
\begin{remark}
The image of $\PP_3$ via $|I_C(3)|$ is known to be the \textit{Donagi-Izadi cubic threefold} in $\PP_4^*$ \cite{Donagi}. This is singular with a nodal point $P=E_0|_C$.
\end{remark}

Let $\pi_P$ be the projection from $\PP_4^*$ at $P$ and then we have the following commutative diagram,
\begin{equation}
\xymatrix{ \overline{M}(2) \ar@{-->}[r]^{\Phi}\ar[d]_{\Psi}^{\wr} & \Ww^2 \ar@{-->}[d]^{\pi_P} \ar@{-->}[ld]^{\tau} \\
        \PP_3 \ar[r]_{f_Q}^{\sim} & \PP_3^*,}
\end{equation}
where the map $f_Q$ is defined as follows: Let $H'$ be a hyperplane in $\PP_3^*$ and then it pulls back via $\pi_P$ to a hyperplane in $\PP_4^*$ containing $P$, and again to $Q$ and a residual hyperplane $H\subset \PP_3$ by $\Phi$.

Let $p_E$ be a point in $\PP_3$ corresponding to $E\in \overline{M}(2)$ and $H'_E$ be its hyperplane in $\PP_3^*$. As above, we can assign a residual hyperplane $H_E\subset \PP_3$ and a conic $C_E= H_E\cap Q$ to $E$. Simply, $f_Q$ is a polar map given by
\begin{equation}
[x_0, \cdots, x_3] \mapsto [\frac{\partial f}{\partial t_0}(x), \cdots, \frac{\partial f}{\partial t_3}(x)],
\end{equation}
where $f$ is the homogeneous polynomial of degree 2 defining $Q$. The hyperplane $H_E\subset \PP_3$, corresponding to $f_Q (p_E)$, is given by the equation,
\begin{equation}
\sum_{i=0}^3 \frac{\partial f}{\partial t_i}(p_E)t_i=0,
\end{equation}
and, from the Euler formula, it is clear that $p_E\in Q$ is equivalent to $p_E\in H_E$.
Assume that $p_E \not \in Q$. Let us recall that $C_E=H_E\cap Q$ is the set of points $p\in Q$ for which $p_E$ is contained in the tangent plane of $Q$ at $p$. In particular, $E$ is fitted into an extension (\ref{ext}) where $Z$ is a point $p$ with multiplicity $2$. In other words, there exists a section $s$ of $E$ whose zero is $p$ with multiplicity $2$. We can have the same argument for the case when $p_E\in Q$.

\begin{proposition}
Let $E$ be a stable sheaf in $\overline{M}(2)$. The set of points $p\in Q$ which is the zero with multiplicity 2 for some section of $E$, forms a conic $C_E$ in $Q$. This defines an isomorphism
$$f_Q \circ \Psi : \overline{M}(2) \rightarrow  \PP_3^*.$$
\end{proposition}

\bibliographystyle{amsplain}
\bibliography{quadric}
\end{document}